\documentclass{amsart}
\usepackage{amsfonts}
\usepackage{amsmath}
\usepackage{amscd}
\usepackage{amssymb}
\usepackage{graphicx}
\usepackage{indentfirst}
\usepackage{latexsym}
\usepackage{fancyhdr}
\usepackage{enumerate}

\newcommand\N{\mathbb{N}}
\newcommand\C{\mathbb{C}}
\newcommand\R{\mathbb{R}}

\newtheorem{theorem}{Theorem}
\newtheorem{lemma}{Lemma}
\newtheorem{corollary}{Corollary}
\newtheorem{conjecture}{Conjecture}

\theoremstyle{remark}

\theoremstyle{definition}

\newcommand{\be}{\begin{equation}}
\newcommand{\ee}{\end{equation}}
\newcommand{\bea}{\begin{eqnarray}}
\newcommand{\eea}{\end{eqnarray}}
\newcommand{\ben}{\begin{eqnarray*}}
\newcommand{\een}{\end{eqnarray*}}
\newcommand{\bet}{\begin{equation} \begin{split}}
\newcommand{\eet}{\end{split} \end{equation}}

\begin{document}
\title[]{Irreducibility of Chebyshev-Lissajous polynomials}
\author{Hanxiong Zhang}
\address{School of Sciences\\China University of Mining and Technology\\Beijing, 100083, China}
\email{zhanghx@cumtb.edu.cn, zhanghanxiong@163.com}

\begin{abstract}
We study certain kind of polynomials associated with Lissajous curves, called Chebyshev-Lissajous polynomials.
We investigate their irreducibilities over the real numbers and complex numbers, thus comfirming two conjectures proposed by Merino\cite{Merino2003}.
\end{abstract}

\maketitle

\section{Introduction}
Suppose $m$ and $n$ are two coprime positive integers, $a$ and $b$ are two real numbers. The plane curve given by $$x=\cos (mt+a),\quad y=\cos (nt+b)$$
is called a Lissajous  curve, or Lissajous figure. It is a  beautiful pattern formed when two harmonic vibrations along perpendicular lines are superimposed.
After eliminating the parameter $t$, we can get a polynomial equation
$f(x,y)=0$. Such  polynomial $f(x,y)$ is called a Lissajous polynomial.

For any natural number $n$, the $n^{th}$ Chebyshev polynomial of the first kind $T_n(x)$, is determined by
$$T_n(\cos\theta) = \cos n\theta,\quad \forall \theta.$$
For example,
$T_0(x)=1, T_1(x)=x, T_2(x)=2x^2-1, T_3(x)=4x^3-3x.$
By the Sum-to-Product formula
$$\cos(n+1)\theta +\cos(n-1)\theta = 2\cos\theta\cos n\theta,$$
 we can get the recursive relation  for Chebyshev polynomials:
$$T_{n+1}(x) = 2x T_n(x) -T_{n-1}(x), \quad \forall n\geq 1.$$
By induction, it is not hard to see that $T_n(x)$ is a polynomial of degree $n$ with integer coefficients,  and the coefficient of the highest term is $2^{n-1}$.
When $n$ is odd,  $T_n(x)$ is an odd function. When $n$ is even,  $T_n(x)$ is an even function.
Chebyshev polynomials also satisfy the following nested property:
$$T_m(T_n(x)) = T_{mn}(x), \quad \forall m,n \in \N^+.$$

How are Lissajous curves and Chebyshev polynomials related to each other? Here is a simple example: Suppose $m,n$ are two coprime positive integers and $m$ is odd. Consider the Lissajous curve
given by
$$x=\cos mt, \quad y=\sin nt = \cos\big(nt-\frac{\pi}{2}\big).$$ Using Chebyshev polynomials, we have
\begin{align*}T^2_n(x) + T^2_m(y)
&=T^2_n(\cos mt) + T^2_m \Big(\cos \big(nt - \frac{\pi}{2} \big)\Big)=\cos^2 mnt + \cos^2 \Big( mnt - \frac{m\pi}{2}\Big)\\
&= \cos^2 mnt+ \sin^2 mnt =1.
\end{align*}
Hence, $T^2_n(x) + T^2_m(y) -1$ is a Lissajous polynomial. Since it is given by Chebyshev polynomials, we call it a Chebyshev-Lissajous  polynomial.
For example, $T^2_1(x) + T^2_1(y) -1= x^2 + y^2 -1$, and the associated Lissajous curve is just the unit circle.
Another example,  $T^2_2(x) + T^2_3(y) -1 = (2x^2 -1)^2 + (4y^3 -3y)^2-1$, and the associated Lissajous curve is given in Figure \ref{figure23}.

\begin{figure}[h]\label{figure23}
\centering\includegraphics[height=5cm]{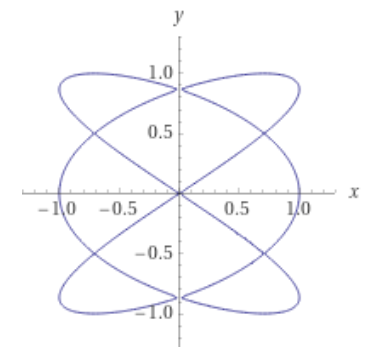}
\caption{Curve $T^2_2(x) + T^2_3(y) -1 =0$}
\end{figure}

 Merino\cite{Merino2003} introduced the general procedure of using Chebyshev polynomials to get the equations of Lissajous curves
 (we have slightly modified the description):
Suppose $m$ and $n$ are two coprime positive integers,   $a$ and $b$ are two real numbers, and $\delta = mb-na$.  Consider the Lissajous
curve given by $x=\cos (mt+a), y=\cos (nt+b)$, then
$$T_n(x) = \cos(mnt+ na), \quad T_m(y) = \cos(mnt+ mb) =\cos(mnt+ na +\delta).$$
If $\sin\delta= 0$, i.e.,  $\delta$ is an integral multiple of $\pi$, then
$$ T_m(y) =\pm \cos(mnt +na) = \pm T_n(x).$$
This is a degenerate  Lissajous curve.  If $\sin\delta\neq 0$, by eliminating the parameter $t$, we can get
$$T^2_n(x) -2T_n(x) T_m(y) \cos\delta + T^2_m(y) -\sin^2\delta =0.$$
This is a non-degenerate  Lissajous curve. Thus, a Lissajous curve is degenerate or not, depends on whether $\delta$ is an integral multiple of $\pi$.

\begin{figure}[h]\label{figure33}
\centering\includegraphics[height=5cm]{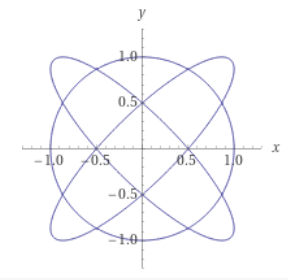}
\caption{Curve $T^2_3(x) + T^2_3(y) -1 =0$}
\end{figure}

Obviously, the above two equations are well-defined for any positive integers $m$ and $n$. For example, the curve defined by $T^2_3(x) + T^2_3(y)-1=0$
is the union of the unit circle and two ellipses, see Figure \ref{figure33}. In the end of \cite{Merino2003}, Merino proposed two conjectures:

\begin{conjecture}\label{conj1} Suppose $m$ and $n$ are two positive integers,  $\delta$ is a real number with $\sin\delta \neq 0$, then the curve defined by
$$T^2_n(x) -2T_n(x) T_m(y) \cos\delta + T^2_m(y) -\sin^2\delta =0$$
is the union of a finite number of Lissajous curves.
\end{conjecture}

\begin{conjecture}\label{conj2} Suppose $m$ and $n$ are two positive integers,  $\delta$ is a real number with $\sin\delta \neq 0$, then
the polynomial $$T^2_n(x) -2T_n(x) T_m(y) \cos\delta + T^2_m(y) -\sin^2\delta$$
is irreducible over $\mathbb{R}$ if and only if $m$ and $n$ are coprime.
\end{conjecture}

From the physical point of view, these two conjectures are quite obvious.  The main purpose of this article is to prove these two conjectures, mathematically.
To achieve this, we need some theorems.

\begin{theorem}\label{theorem1}
Suppose $n$ is a positive integer and $\delta$ is a real number,  then
\begin{align*}& T^2_n(x) -2T_n(x) T_n(y) \cos\delta + T^2_n(y) -\sin^2\delta\\
 =&4^{n-1}\prod_{k=1}^{n} \Big(x^2-2\cos\frac{\delta+2k\pi}{n}xy+y^2-\sin^2\frac{\delta+2k\pi}{n}\Big).\end{align*}
\end{theorem}
This theorem tells us: when $\sin\delta\neq 0$, the curve defined by
$$T^2_n(x) -2T_n(x) T_n(y) \cos\delta + T^2_n(y) -\sin^2\delta =0$$
is the union of $n$ ellipses (including circles).

When $\delta = \frac{\pi}{2}$, we get an important special case.
\begin{corollary}\label{cor}
Suppose $n$ is a positve integer,  then
\begin{align*} T^2_n(x)+ T^2_n(y) -1 = 4^{n-1}\prod_{k=1}^{n} \Big(x^2-2\cos\frac{(4k+1)\pi}{2n}xy+y^2-\sin^2\frac{(4k+1)\pi}{2n}\Big).\end{align*}
\end{corollary}

\begin{theorem}\label{theorem2}
Suppose $n$ is a positve integer, then
\begin{align*} T_{2n+1}(x) - T_{2n+1}(y)
 =4^{n}(x-y)\prod_{k=1}^{n} \Big(x^2-2\cos\frac{2k\pi}{2n+1}xy+y^2-\sin^2\frac{2k\pi}{2n+1}\Big).\end{align*}
\end{theorem}

\section{Preliminary}
First, let us introduce a factorization of Chebyshev polynomials. For more factorizations, see \cite{wolfram2022}.
\begin{lemma}\label{lemma1}
Suppose $n$ is a positve integer and $\delta$ is a real number, then
$$T_n(x) - \cos\delta= 2^{n-1}\prod_{k=1}^{n}  \Big(x - \cos\frac{\delta + 2k\pi}{n}\Big).$$
\end{lemma}

\begin{proof}
Since $T_n(\cos\theta)=\cos n\theta,$ we have
$$T_n \Big(\cos\frac{\delta + 2k\pi}{n}\Big)=\cos(\delta + 2k\pi)=\cos\delta, \quad k=1,2,\cdots,n. $$
When $\sin\delta \neq 0$, i.e., $\delta$ is not an integral multiple of $\pi$,
$$\displaystyle \cos\frac{\delta+ 2k\pi}{n}, \quad k= 1,2,\cdots,n$$ are mutually different.
Thus $\displaystyle \cos\frac{\delta+ 2k\pi}{n} (k=1,2,\cdots,n)$ are different roots of $T_n(x)-\cos\delta$.
Since $T_n(x)-\cos\delta$ is a polynomial of degree $n$, with highest coefficient $2^{n-1}$, hence
$$T_n(x) - \cos\delta= 2^{n-1}\prod_{k=1}^{n}  \Big(x - \cos\frac{\delta + 2k\pi}{n}\Big).$$
When $\sin\delta =0$,  by a suitable limit process, it is easy to see that the above equality still holds for any real number $x$.
Hence the two sides are equal as polynomials of $x$.
\end{proof}


\begin{corollary}\label{cor2}
Suppose $n$ is a positve integer and $\delta$ is a real number,  then
$$\cos n\theta - \cos\delta = 2^{n-1}\prod_{k=1}^{n}   \Big(\cos\theta-\cos\frac{\delta + 2k\pi}{n} \Big).$$
\end{corollary}

\begin{lemma}\label{zerolemma}
Suppose $f(x, y)$ is a polynomial of two variables with real coefficients. If there exists  $r>0$, such that
$$f(x, y) =0, \quad \forall x,y \in (-r,r),$$
then $f = 0$.
\end{lemma}
\begin{proof}
Suppose $\displaystyle f(x,y) = \sum_{i,j} a_{i,j} x^i y^j$, which is a finite sum.
Let $x=y=0$,  we get $a_{0,0} =0$. Take partial derivative with respect to $x$ and let $x=y=0$, we get $a_{1,0}=0$.
By similar operations, we can see that $a_{i,j} = 0$ for all $i$ and $j$. Hence $f=0$.
\end{proof}

\begin{corollary}\label{identical_cor}
Suppose$f(x, y)$ and $g(x,y)$ are polynomials of two variables with real coefficients. If there exists $r>0$, such that
$$f(x, y) =g(x,y), \quad \forall x,y \in (-r,r),$$
then $f = g$.
\end{corollary}

\section{Proof of Theorem \ref{theorem1}}
\begin{proof}
When $x,y\in [-1,1]$, let $x=\cos\alpha, y=\cos\beta$, then
\begin{align*}
&T^2_n(x) -2T_n(x) T_n(y) \cos\delta + T^2_n(y) -\sin^2\delta\\
=& \cos^2 n\alpha-2\cos n\alpha \cos n\beta \cos\delta +\cos^2 n\beta -\sin^2\delta \\
=&\frac{1+\cos 2n\alpha}{2}- \big[\cos n(\alpha+\beta) +\cos n(\alpha-\beta) \big]\cos\delta +\frac{1+\cos 2n\beta}{2}-\sin^2\delta\\
=&  \cos n(\alpha+\beta) \cos n(\alpha-\beta) -\big[\cos n(\alpha+\beta) +\cos n(\alpha-\beta) \big]\cos\delta + \cos^2\delta\\
=& \big[  \cos n(\alpha+\beta) -\cos\delta  \big]\cdot \big[  \cos n(\alpha-\beta) -\cos\delta  \big].
\end{align*}
By Corollary \ref{cor2}, the above is equal to
\begin{align*}2^{n-1}\prod_{k=1}^n \Big(\cos(\alpha+\beta)- \cos\frac{\delta + 2k\pi}{n}\Big) \cdot 2^{n-1}\prod_{k=1}^n \Big(\cos(\alpha-\beta)-  \cos\frac{\delta + 2k\pi}{n}\Big).
\end{align*}
Since $\cos(\alpha \pm \beta) = \cos\alpha \cos\beta \mp \sin\alpha\sin\beta$, the above is equal to
\begin{align*}
&4^{n-1}\prod_{k=1}^n \Big(\cos\alpha \cos\beta - \sin\alpha\sin\beta-  \cos\frac{\delta + 2k\pi}{n} \Big)
 \Big(\cos\alpha \cos\beta + \sin\alpha\sin\beta-  \cos\frac{\delta + 2k\pi}{n} \Big)\\
=& 4^{n-1}\prod_{k=1}^n \Big[ \Big(\cos\alpha \cos\beta -  \cos\frac{\delta + 2k\pi}{n} \Big)^2 -  \Big(\sin\alpha\sin\beta \Big)^2 \Big]\\
=& 4^{n-1}\prod_{k=1}^n \Big[ \Big(xy -  \cos\frac{\delta + 2k\pi}{n} \Big)^2 - (1-x^2)(1-y^2) \Big]\\
=&4^{n-1}\prod_{k=1}^n\Big(x^2-2 \cos\frac{\delta + 2k\pi}{n} xy+y^2-\sin^2\frac{\delta + 2k\pi}{n} \Big).
\end{align*}
Hence, when $x,y\in [-1,1]$, we have
\begin{align*}
&T^2_n(x) -2T_n(x) T_n(y) \cos\delta + T^2_n(y) -\sin^2\delta\\
=&4^{n-1}\prod_{k=1}^n\Big(x^2-2 \cos\frac{\delta + 2k\pi}{n} xy+y^2-\sin^2\frac{\delta + 2k\pi}{n} \Big).
\end{align*}
By Corollary \ref{identical_cor},  the two sides are equal as polynomials. This completes the proof of Theorem \ref{theorem1}.
\end{proof}

\section{Proof of Theorem \ref{theorem2}}
\begin{proof}
When $y\in [-1,1]$, let $y=\cos\beta$. By Lemma \ref{lemma1},
\begin{align*}
&  T_{2n+1}(x) - T_{2n+1}(y) =T_{2n+1}(x) - \cos(2n+1)\beta = 2^{2n}\prod_{k=1}^{2n+1}  \Big(x - \cos (\beta +\frac{2k\pi}{2n+1} )\Big)\\
=& 4^n (x-\cos\beta) \prod_{k=1}^{n} \Big(x - \cos (\beta +\frac{2k\pi}{2n+1} )\Big)\Big(x - \cos (\beta -\frac{2k\pi}{2n+1} )\Big)\\
=& 4^n (x-\cos\beta)\prod_{k=1}^n   \Big[ \Big(x - \cos\beta \cos\frac{2k\pi}{2n+1} \Big)^2 -  \Big(\sin\beta \sin\frac{2k\pi}{2n+1} \Big)^2 \Big]\\
=& 4^n (x-y)\prod_{k=1}^n   \Big[ \Big(x - y \cos\frac{2k\pi}{2n+1} \Big)^2 -  (1-y^2)  \sin^2 \frac{2k\pi}{2n+1} \Big]\\
=& 4^n (x-y) \prod_{k=1}^n\Big(x^2-2 \cos\frac{2k\pi}{2n+1} xy+y^2-\sin^2\frac{2k\pi}{2n+1} \Big).
\end{align*}
Hence, for any $x, y\in [-1,1]$, we have
\begin{align*} T_{2n+1}(x) - T_{2n+1}(y)
 =4^{n}(x-y)\prod_{k=1}^{n} \Big(x^2-2\cos\frac{2k\pi}{2n+1}xy+y^2-\sin^2\frac{2k\pi}{2n+1}\Big).\end{align*}
By Corollary \ref{identical_cor}, the two sides are equal as polynomials. This completes the proof of Theorem \ref{theorem2}.
\end{proof}

\section{Proof of the sufficient part of Conjecture \ref{conj2}}
We rewrite the the sufficient part of Conjecture \ref{conj2} as a theorem.
\begin{theorem}\label{sufficient_part}
Suppose $m$ and $n$ are two coprime positive integers,  and $\delta$ is a real number with $\sin\delta \neq 0$, then
$$T^2_n(x) -2T_n(x) T_m(y) \cos\delta + T^2_m(y) -\sin^2\delta$$
is irreducible over $\R$.
\end{theorem}

Tverberg\cite{Tverberg1964} provided a simple proof of the following result:  Suppose$f(x)$ and $g(y)$ are two nonconstant polynomials with complex coefficients.
If the degrees of $f,g$ are coprime, then $f(x) + g(y)$ is irreducible over $\C$.

The first half of the following proof is essentially the same as the proof in \cite{Tverberg1964}, while the second half is new.

\begin{proof}  Suppose $T^2_n(x) -2T_n(x) T_m(y) \cos\delta + T^2_m(y) -\sin^2\delta= P(x,y)Q(x,y)$,
where $P, Q \in \R[x,y]$, and $P,Q$ are both nonconstant, then
$$T^2_n(x^m) -2T_n(x^m) T_m(y^n) \cos\delta + T^2_m(y^n) -\sin^2\delta = P(x^m,y^n)Q(x^m,y^n).$$
Denote $P_1, Q_1$  the homogeneous parts of $P(x^m,y^n), Q(x^m,y^n)$ with highest degrees, then
$$4^{n-1} x^{2mn}-2^{m+n-1}\cos\delta\cdot x^{mn}y^{mn} +4^{m-1} y^{2mn} = P_1 Q_1. $$
Let $w= \cos \delta + i \sin\delta$, then the left hand side equals
$$(2^{n-1} x^{mn} -w 2^{m-1}y^{mn}) (2^{n-1} x^{mn} -\bar{w} 2^{m-1}y^{mn}), $$
which has no monomial factors. Hence, $P_1$ has at least two terms, i.e.,
$$P_1 = p x^{km}y^{ln} +p' x^{k'm}y^{l'n}+\cdots.$$
As $P_1$ is homogeneous, $km+ln = k'm+l'n$, hence $(k-k') m = (l'-l)n$. Since $m,n$ are coprime, we have $n| k-k', m| l'-l$.
Suppose $k>k'$, then $k \geq n+k' \geq n$, $l'>l$, and $l'\geq m +l\geq m$. So $\deg P_1= km+ln\geq mn,$
with equality if and only if $k=n, l'=m, l=k'=0$.

Similarly, $Q_1$ has at least two terms, i.e.,
$$Q_1= q x^{sm}y^{tn} +q' x^{s'm}y^{t'n}+\cdots.$$
As $Q_1$ is homogeneous, $sm+tn = s'm+t'n$, $(s-s') m = (t'-t)n$.
Since $m,n$ are coprime, we have $n| s-s', m| t'-t$. Suppose $s>s'$, then $s \geq n+s' \geq n$, $t'>t$, and $t'\geq m +t\geq m$. So $\deg Q_1 = sm+tn\geq mn,$
with equality if and only if  $s=n, t'=m, t=s'=0$.

Since the degree of $4^{n-1} x^{2mn}-2^{m+n-1}\cos\delta\cdot x^{mn}y^{mn} +4^{m-1} y^{2mn} $ is $2mn$,
 the above inequalities are all equalities, and $P_1, Q_1$ both have two terms, i.e., $P_1 = p x^{mn} +p' y^{mn}, Q_1= q x^{mn} +q' y^{mn}$.
 Compare the coefficients of the two sides of
$$4^{n-1} x^{2mn}-2^{m+n-1}\cos\delta\cdot x^{mn}y^{mn} +4^{m-1} y^{2mn}  = (p x^{mn} +p' y^{mn}) (q x^{mn} +q' y^{mn}), $$
we have
$pq = 4^{n-1},p'q'=4^{m-1}, pq' + p'q =-2^{m+n-1}\cos\delta.$
Hence,
$$-2^{m+n-1}\cos\delta= p\cdot \frac{4^{m-1}}{p'} +p'\cdot \frac{4^{n-1}}{p}. $$
Since $\displaystyle \frac{p}{p'}$ and $\displaystyle \frac{p'}{p}$ are both positive or both negative, we can take absolute values and get
$$2^{m+n-1}|\cos\delta| = \Big|\frac{4^{m-1}p}{p'}\Big| + \Big|\frac{4^{n-1}p'}{p}\Big|
\geq 2\sqrt{\Big|\frac{4^{m-1}p}{p'}\Big| \cdot \Big|\frac{4^{n-1}p'}{p}\Big| } =2^{m+n-1}.$$
This implies $|\cos\delta| = 1$, which contradicts with $\sin\delta\neq 0$.\end{proof}

We can strengthen the conclusion a little further.
\begin{theorem}
Suppose $m$ and $n$ are two coprime positive integers,  and $\delta$ is a real number with $\sin\delta \neq 0$, then
$$T^2_n(x) -2T_n(x) T_m(y) \cos\delta + T^2_m(y) -\sin^2\delta$$
is irreducible over $\C$.
\end{theorem}

\begin{proof} Suppose $T^2_n(x) -2T_n(x) T_m(y) \cos\delta + T^2_m(y) -\sin^2\delta = f_1 f_2\cdots f_l$,
where $l\geq 2,  f_1, f_2,\cdots, f_l$ are all nonconstant irredubible polynomials in $\C[x,y]$.
Take conjugation, we get $T^2_n(x) -2T_n(x) T_m(y) \cos\delta + T^2_m(y) -\sin^2\delta = \bar{f}_1 \bar{f}_2\cdots \bar{f}_l$.
Since the highest term (in $x$) of the left hand side is $4^{n-1}x^{2n}$, we can assume that each of $f_1, f_2,\cdots, f_l$ has highest coefficient $\sqrt[l]{4^{n-1}}$.

As $\C[x,y]$ is a UFD,
$\bar{f}_1, \bar{f}_2, \cdots, \bar{f}_l$ must be a rearrangement of $f_1, f_2,\cdots, f_l$.
By Theorem \ref{sufficient_part},  $T^2_n(x) -2T_n(x) T_m(y) \cos\delta + T^2_m(y) -\sin^2\delta$ is irreducible over $\R$,
we have $l=2$ and $f_2 = \bar{f}_1 $. Hence,
$$T^2_n(x) -2T_n(x) T_m(y) \cos\delta + T^2_m(y) -\sin^2\delta = f_1(x,y) \bar{f}_1(x,y).$$
Take $x = \frac{\pi}{2n}$ and $y = \frac{\pi}{2m}$, the left hand side is equal to $-\sin^2\delta<0$, while the right hand side is nonnegative, contradiction!
\end{proof}

When $\delta= \frac{\pi}{2}$, we get an important special case.
\begin{corollary}
Suppose $m$ and $n$ are two coprime positive integers,  then $T^2_n(x) + T^2_m(y) -1$ is irreducible over $\C$.
\end{corollary}

\section{Proof of the necessary part of Conjecture \ref{conj2}}
Denote $t_n(x,y) = T^2_n(x) -2T_n(x) T_n(y) \cos\delta + T^2_n(y) -\sin^2\delta$, then Theorem \ref{theorem1} tells us:  when $n>1$, $t_n(x,y)$ is reducible over $\mathbb{R}$.
Now we rewrite the necessary part of Conjecture \ref{conj2} as a theorem.
\begin{theorem}Suppose $m$ and $n$ are two positive integers, and $\delta$ is a real number with $\sin\delta \neq 0$.
If $T^2_n(x) -2T_n(x) T_m(y) \cos\delta + T^2_m(y) -\sin^2\delta$ is irreducible over $\mathbb{R}$, then $m,n$ are coprime. \end{theorem}

\begin{proof}We use contradiction. Suppose $d= \gcd(m,n)>1$, let $m = m'd, n=n'd$, then
\begin{align*}
&T^2_n(x) -2T_n(x) T_m(y) \cos\delta + T^2_m(y) -\sin^2\delta\\
=& T_d(T_{n'}(x))^2 -2T_d(T_{n'}(x))T_d(T_{m'}(y)) \cos\delta +T_d(T_{m'}(y))^2 -\sin^2\delta \\
=& t_d(T_{n'}(x), T_{m'}(y))
\end{align*}
is reducible over $\mathbb{R}$.
\end{proof}

\section{Proof of Conjecture \ref{conj1}}
\begin{proof}
Let $d =\gcd(m,n)$, and $m = m'd, n=n'd$. Using the notations introduced in the last section and Theorem \ref{theorem1}, we have
\begin{align*}
&T^2_n(x) -2T_n(x) T_m(y) \cos\delta + T^2_m(y) -\sin^2\delta= t_d(T_{n'}(x), T_{m'}(y))\\
=&4^{d-1}\prod_{k=1}^d\Big(T^2_{n'}(x)-2 \cos\frac{\delta + 2k\pi}{d} T_{n'}(x)T_{m'}(y)+T_{m'}(y)^2-\sin^2\frac{\delta + 2k\pi}{d} \Big).\end{align*}
As $m'$ and $n'$ are coprime, each factor
$$T^2_{n'}(x)-2 \cos\frac{\delta + 2k\pi}{d} T_{n'}(x)T_{m'}(y)+T_{m'}(y)^2-\sin^2\frac{\delta + 2k\pi}{d} $$
is irreducible over $\C$, corresponding to a Lissajous curve. Hence, the curve defined by $T^2_n(x) -2T_n(x) T_m(y) \cos\delta + T^2_m(y) -\sin^2\delta =0$
is the union of a finite number of Lissajous curves.
\end{proof}

\section{The degenerate case of Conjecture \ref{conj2}}
The two conjectures proposed in \cite{Merino2003}, also include the degenerate case, i.e.,  when the Chebyshev-Lissajous polynomial is $T_n(x)\pm T_m(y)$.
In this section, we consider the degenerate case of Conjecture \ref{conj2}, i.e., when is the polynomial
$T_n(x) \pm T_m(y)$ irreducible?  The answer is a little bit subtle.

\begin{theorem} Suppose $m$ and $n$  are two positive integers, $d=\gcd(m,n)$,  then
$T_n(x)- T_m(y)$ is irreducible over $\mathbb{R}$ if and only if $d=1$, and $T_n(x) + T_m(y)$ is irreducible over $\mathbb{R}$ if and only if $d\leq 2$.
\end{theorem}

If we replace $\R$ by $\C$, the conclusion is also correct.

\begin{proof}
When $d=1$, by the conclusion of \cite{Tverberg1964}, $T_n(x) \pm T_m(y)$ is irreducible over $\C$, hence  irreducible over $\R$.

When $d>1$, let $m = m'd, n=n'd$, then $m',n'$ are coprime.

1. The case of $T_n(x) - T_m(y)$ is simple.  Since $x-y | T_d(x) - T_d(y)$, $T_d(x) - T_d(y)$ is reducible over $\R$, so
$T_n(x) - T_m(y) =T_d(T_{n'}(x))-T_d(T_{m'}(y))$ is reducible over $\R$.

2. The case of $T_n(x) + T_m(y)$ is a little complicated.

(1)When $d$ is odd, since $x+y | T_d(x) + T_d(y)$, $T_d(x) + T_d(y)$ is reducible over $\R$, hence
$T_n(x) + T_m(y) =T_d(T_{n'}(x))+T_d(T_{m'}(y))$ is reducible over $\R$.

(2)When $d=2$, $T_n(x) + T_m(y) = T_{2n'}(x) + T_{2m'}(y) = 2\big(T_{n'}^2(x)+ T_{m'}^2(y)-1\big)$  is irreducible over $\R$.
The simplest example is $T_2(x) +T_2(y) = 2(x^2 + y^2-1)$.

(3)When $d$ is an even integer bigger than $2$,  let $d=2e$, then $e>1$. By Corollary \ref{cor}, $T_d(x) + T_d(y)= 2\big( T_e^2(x) +T_e^2(y) -1 \big)$  is reducible over  $\R$.
Hence, $T_n(x) + T_m(y) =T_d(T_{n'}(x))+T_d(T_{m'}(y))$  is reducible over  $\R$.
\end{proof}

\section{The degenerate case of Conjecture \ref{conj1}}
In this section, we consider the degenerate case of Conjecture \ref{conj1}, i.e., whether the curve defined by $T_n(x) \pm T_m(y) =0$
is the union of a finite number of  Lissajous curves?  The answer is yes.
\begin{theorem} Suppose $m$ and $n$ are two positive integers, then the curve defined by
$T_n(x)\pm T_m(y) =0$ is the union of a finite number of Lissajous curves.
\end{theorem}

\begin{proof}
If $m$ and $n$ are coprime, then $T_n(x) \pm T_m(y)$ is irreducible over $\C$, so the curve defined by $T_n(x) \pm T_m(y)=0$ is a Lissajous curve.

If $m$ and $n$ are not coprime, then $d =\gcd(m,n)>1$. Let $m = m'd, n=n'd$.

(1)If $d$ is odd, suppose $d =2e+1$. By Theorem \ref{theorem2},
\begin{align*}& T_n(x) - T_m(y) = T_d( T_{n'}(x)) - T_d(T_{m'}(y))\\
=&4^{e}\big(T_{n'}(x) -T_{m'}(y)\big)\prod_{k=1}^{e} \Big(T_{n'}(x)^2-2\cos\frac{2k\pi}{2e+1}T_{n'}(x)T_{m'}(y)+T_{m'}(y)^2-\sin^2\frac{2k\pi}{2e+1}\Big).\end{align*}
As $m'$ and $n'$ are coprime, each factor on the right hand side is irreducible, corresponding to a Lissajous curve.
Hence, the curve defined by $T_n(x) - T_m(y) =0$ is the union of a finite number of Lissajous curves.

When $d$ is odd,  $T_d(x)$ is an odd function, so
\begin{align*}& T_n(x) + T_m(y) = T_d( T_{n'}(x)) + T_d(T_{m'}(y))= T_d( T_{n'}(x)) - T_d(-T_{m'}(y))\\
=&4^{e}\big(T_{n'}(x) + T_{m'}(y)\big)\prod_{k=1}^{e} \Big(T_{n'}(x)^2+2\cos\frac{2k\pi}{2e+1}T_{n'}(x)T_{m'}(y)+T_{m'}(y)^2-\sin^2\frac{2k\pi}{2e+1}\Big).\end{align*}
Hence, the curve defined by $T_n(x) + T_m(y) =0$ is the union of a finite number of Lissajous curves.

(2)When $d$ is even, suppose $d =2e$. By $T_d(x) + T_d(y) = 2\big(T_e^2(x) + T_e^2(y)-1 \big)$, we have
\begin{align*}
&T_n(x)  + T_m(y) = T_d( T_{n'}(x)) + T_d(T_{m'}(y)) = 2\big(T_e^2(T_{n'}(x)) + T_e^2(T_{m'}(y))-1 \big)\\
=&2 \cdot 4^{e-1}\prod_{k=1}^e  \Big(T^2_{n'}(x)-2 \cos\frac{(4k+1)\pi}{2e} T_{n'}(x)T_{m'}(y)+T_{m'}(y)^2-\sin^2\frac{(4k+1)\pi}{2e} \Big).\end{align*}
Hence, the curve defined by $T_n(x) + T_m(y) =0$ is the union of a finite number of Lissajous curves.

When $d$ is even, both $m$ and $n$ are even. Using the identity
$T_n(x) - T_m(y) = 2\Big(  T_{\frac{n}{2}}(x) - T_{\frac{m}{2}}(y) \Big) \Big(  T_{\frac{n}{2}}(x) + T_{\frac{m}{2}}(y) \Big)$
and induction, we can see that the curve defined by $T_n(x) - T_m(y) =0$ is the union of a finite number of Lissajous curves.
\end{proof}

\section{Acknowledgement}
The author thanks Dr. Kai-Liang Lin for his encouragement and many helpful discussions.

\end{document}